\documentclass{article}
\usepackage[%
journal=XXX,    
lang=british,   
]{ems_journal}

\numberwithin{equation}{section}

\theoremstyle{plain}
\newtheorem{theorem}{Theorem}[section]
\newtheorem{lemma}[theorem]{Lemma}
\theoremstyle{definition}
\newtheorem{definition}{Definition} 
\newtheorem{example}[theorem]{Example}

\def\opn#1#2{\def#1{\operatorname{#2}}} 
\opn\chara{char} \opn\length{\ell} \opn\pd{pd} \opn\rk{rk}
\opn\projdim{proj\,dim} \opn\injdim{inj\,dim} \opn\rank{rank}
\opn\depth{depth} \opn\Min{Min} \opn\supp{supp} \opn\grade{grade} \opn\height{height}
\opn\embdim{emb\,dim} \opn\codim{codim}

\opn\Tr{Tr} \opn\bigrank{big\,rank}
\opn\superheight{superheight}\opn\lcm{lcm}
\opn\trdeg{tr\,deg}
\opn\reg{reg} \opn\lreg{lreg} \opn\ini{in} \opn\lpd{lpd}
\opn\size{size}\opn\bigsize{bigsize}
\opn\cosize{cosize}\opn\bigcosize{bigcosize}
\opn\sdepth{sdepth}\opn\sreg{sreg}
\opn\link{link}\opn\fdepth{fdepth}

\begin{document}

\title{Some Properties of Path and Closed Neighborhood Ideals}
\titlemark{Some Properties of Path and Closed Neighborhood Ideals 
    }


%

\emsauthor{1}{
	\givenname{Hafsa}
	\surname{Bibi}
	\orcid{0009-0001-3042-8871}}{H.~Bibi}
\emsauthor{2}{
	\givenname{Azhar}
	\surname{Farooq}
	\mrid{}
	\zblid{}
	\orcid{0009-0009-8710-4673}}{A.~Farooq}
\emsauthor{3}{
	\givenname{Hanni}
	\surname{Garminia}
	\mrid{}
	\zblid{}
	\orcid{0009-0001-8949-4961}}{H.~Garminia}
\emsauthor{4}{
	\givenname{Irawati}
	\mrid{}
	\zblid{}
	\orcid{0009-0001-7167-0269}}{Irawati}    

\Emsaffil{1}{
	\department{Doctoral Mathematics, Faculty of Mathematics and Natural Sciences}
	\organisation{Institut Teknologi Bandung}
	\rorid{ABC}
	\address{Jalan Ganesa 10}
	\zip{40132}
	\city{Bandung}
	\country{Indonesia}
	\affemail{30124701@mahasiswa.itb.ac.id}}
\Emsaffil{2}{
	\department{1}{Abdus Salam School of Mathematical Sciences}
	\organisation{1}{Government College University}
	\rorid{1}{}
	\address{1}{68-B, New Muslim Town}
	\zip{1}{54600}
	\city{1}{Lahore}
	\country{1}{Pakistan}
	\department{2}{} 
	\organisation{2}{}%
	\rorid{2}{}
	\address{2}{}%
	\zip{2}{}
	\city{2}{}
	\country{2}{} 
	\affemail{2}{azhar.farooq@sms.edu.pk}}
\Emsaffil{3}{
	\department{Algebra Research Group, Faculty of Mathematics and Natural Sciences, Institut Teknologi Bandung}
	\organisation{Institut Teknologi Bandung}
	\rorid{}
	\address{Jalan Ganesa 10}
	\zip{40132}
	\city{Bandung}
	\country{Indonesia}
	\affemail{garminia@itb.ac.id}}
\Emsaffil{4}{
	\department{Algebra Research Group, Faculty of Mathematics and Natural Sciences}
	\organisation{Institut Teknologi Bandung}
	\rorid{}
	\address{Jalan Ganesa 10}
	\zip{40132}
	\city{Bandung}
	\country{Indonesia}
	\affemail{irawati@itb.ac.id}}    

\classification[13F70, 05E40, 13E05, 13B25]{13C15, 13C05, 13-04}


\begin{abstract}
 In this paper, we investigate algebraic properties of two classes of monomial ideals associated with graphs, namely path ideals and closed neighborhood ideals. Our primary focus is on the strong persistence property and several homological invariants, including depth and the $v$-number. We first prove that the path ideal of length $n-1$ associated with a path graph on $n\geq4$ vertices attached to an arbitrary graph satisfies the strong persistence property. We then study closed neighborhood ideals and establish the strong persistence property for graphs containing a leaf vertex. Furthermore, we identify a broader class of graphs satisfying this property, namely those whose vertex set admits a decomposition $V(G)=C\sqcup\{v\}\sqcup X,$ where $C$ is a nonempty clique, every vertex of $C$ is adjacent to each vertex of $X\cup\{v\}$, and the vertex $v$ has no neighbors in $X$. In addition, we investigate algebraic invariants of ideals associated with Helm graphs. Specifically, we compute the depth of the edge ideal and determine the Krull dimension, the depth, and the $v$-number of the closed neighborhood ideal of the Helm graph. These results contribute to a deeper understanding of the interplay between graph-theoretic structures and the algebraic properties of their associated monomial ideals.
\end{abstract}

\maketitle

\section{Introduction}


\indent \indent Let \(K\) be a field and
\(
\mathcal{A}=K[x_1,\dots,x_N]
\)
be the polynomial ring. For a monomial ideal \(I \subseteq \mathcal{A}\), we denote by \(\mathcal{G}(I)\) its unique minimal monomial generating set and by \(\operatorname{Ass}_\mathcal{A}(\mathcal{A}/I)\) the set of associated prime ideals of \(\mathcal{A}/I\).

A classical theorem of Brodmann \cite{Bro79} states that the sequence $\{\operatorname{Ass}_\mathcal{A}(\mathcal{A}/I^k)\}_{k\ge1}$ stabilizes for sufficiently large \(k\). This motivated the study of the persistence property. A monomial ideal \(I\) is said to satisfy the \emph{persistence property} if
\[
\operatorname{Ass}_\mathcal{A}(\mathcal{A}/I^k)\subseteq \operatorname{Ass}_\mathcal{A}(\mathcal{A}/I^{k+1})
\quad \text{for all } k\ge1.
\]
A stronger condition is the \emph{strong persistence property}, defined by
\[
(I^{k+1}:_\mathcal{A}I)=I^k
\quad \text{for all } k\ge1.
\]
It is known that strong persistence implies persistence \cite{HQ15}, while the converse does not hold in general \cite{MMV12}. Furthermore, some ideals do not satisfy even the persistence property, as shown in \cite{MMV12,RNA18}. Several classes of combinatorially defined monomial ideals satisfy the strong persistence property, including edge ideals, cover ideals of perfect graphs, and polymatroidal ideals; see \cite{FHT11,HRV13,HQ15,RT17}. Further classes such as \emph{unisplit ideals}, \emph{separable ideals}, \emph{normal ideals}, and path ideals of centipede-related graphs have also been shown to exhibit this property \cite{RNA18, BRU24}.

The closed neighborhood ideal of a graph was introduced by by Sharifan and Moradi in \cite{Sha+20} and has recently attracted considerable attention. Let \(G=(V(G),E(G))\) be a simple graph with vertex set $V(G)=\{x_1,x_2,\dots,x_n\}.$ For a vertex \(x_i\in V(G)\), the closed neighborhood of \(x_i\) is defined as $N[x_i]=\{x_i\}\cup \{x_j\in V(G):\{x_i,x_j\}\in E(G)\}.$ The closed neighborhood ideal of \(G\), denoted by \(CNI(G)\), is the monomial ideal generated by the products of variables corresponding to the closed neighborhoods of the vertices of \(G\). More precisely,
\[
CNI(G)=\left(\prod_{x_j\in N_G[x_i]}x_j : x_i\in V(G)\right)
\subseteq \mathcal{A},
\]
where $\mathcal{A}=K[x_1,x_2,\dots,x_N]$ is the polynomial ring over a field \(K\). Various algebraic properties of these ideals, including normality, Cohen--Macaulayness, regularity, and persistence properties, have been studied in \cite{Nas+22}, \cite{Nas+25}, \cite{Mor+26}, \cite{Olt+26}.

The \emph{depth} of a finitely generated $\mathcal{A}$-module $M$ is the length of a maximal $M$-regular sequence in $\mathcal{A}$. The study of the depth of monomial ideals and their powers has attracted considerable attention. Brodmann \cite{2Bro79} proved that $\depth(\mathcal{A}/I^n)$ is constant for all sufficiently large $n$. Then $depth(\mathcal{A}/I)$ has been studied by various authors (see for example \cite{Susan}, \cite{MCf}, \cite{Zhu16}).The depth of the edge ideal has been computed for several graph families. In \cite{BRU24}, the authors determined the depth of the edge ideal of the centipede graph, while in \cite{BGIRA} the corresponding result was obtained for the Sunlet graph. Motivated by these works, we compute the depth of the edge ideal of the Helm graph (Theorem~\ref{edge helm}).

The \emph{$v$-number} of a monomial ideal $I$ is defined by
\[
v(I)=\min\left\{\deg(u)\;\middle|\;
u \text{ is a monomial and }(I:u)\in\operatorname{Ass}(\mathcal{A}/I)\right\}.
\]
The $v$-number was introduced by Cooper et al. \cite{vnumber} in connection with the study of Reed--Muller-type codes. Jaramillo and Villarreal \cite{JV} investigated the $v$-number of edge ideals of clutters. More recently, Biswas and Mandal \cite{biswas} obtained explicit formulas for the $v$-number of edge ideals of several classes of graphs, including paths, cycles, the $1$-clique sum of a path and a cycle, the $1$-clique sum of two cycles, and the join of two graphs.

This paper is organized as follows. We investigate persistence properties and homological invariants of path ideals and closed neighborhood ideals associated with graphs. In Section~\ref{SP path attach}, we prove that the path ideal of length \(n-1\) corresponding to a path graph on \(n\ge4\) vertices attached to an arbitrary graph satisfies the strong persistence property (Theorem~\ref{thm path attach}). In Section~\ref{SP of neighborhood ideal}, we establish the strong persistence property for closed neighborhood ideals of graphs containing a vertex of degree one (Theorem~\ref{degree-one-vertex}). We further extend this result to a broader family of graphs whose vertex set admits a decomposition
\[
V(G)=C\sqcup\{v\}\sqcup X,
\]
where \(C\) is a nonempty clique, every vertex of \(C\) is adjacent to each vertex of \(X\cup\{v\}\), and the vertex \(v\) has no neighbors in \(X\) (Theorem~\ref{separable}). Finally, in Section~\ref{helm}, we investigate the edge ideal and the closed neighborhood ideal of the Helm graph. We compute the depth of the edge ideal and the depth of the closed neighborhood ideal (Theorems~\ref{edge helm} and~\ref{closed helm}), determine the \(v\)-number of the closed neighborhood ideal (Theorem~\ref{vnumber helm}). The computational results presented throughout the paper were obtained using the computer algebra systems \textsc{Singular}~\cite{S} and \textsc{Macaulay2}~\cite{Mac}.

\section{Attaching a Path Graph with an Arbitrary Graph } \label{SP path attach}

The path ideals of graphs were first introduced and studied by Conca et al. \cite{con+99} in 1999.
\begin{theorem} \label{thm path attach}
Let $G$ be a simple graph on vertex set $V(G)$ and let $H$ be the graph obtained by attaching a path $P_a$ of length a ($a \ge 4$) by identifying one of its endpoints with a fixed vertex of $G$. 
Let $b$ be a positive integer such that 
\[
1 \le b \le \left\lfloor \frac{a-1}{2} \right\rfloor,
\]
and let $I_b(H)$ be the $b$-path ideal of $H$, i.e., the monomial ideal generated by all squarefree monomials $x_{v_1} x_{v_2} \cdots x_{v_{b+1}}$ for which $v_1, v_2, \dots, v_{b+1}$ form a path of length $b$ in $H$. Then the following hold:

\begin{enumerate}[label=(\roman*)]
        \item  $I_b(H)$ admits a strongly superficial element;
        \item $I_b(H)$ satisfies the persistence property;
        \item $I_b(H)$ satisfies the strong persistence property.
    \end{enumerate}
\end{theorem}

\begin{proof}
Recall that for a monomial ideal, the existence of a strongly superficial element implies both the persistence and the strong persistence properties; see \cite{Nas14, RNA18}. Thus, it suffices to prove statement~(i).

Label the vertices of the attached path $P_a$ consecutively by $1,2,\dots,a$, where the vertex $a$ is identified with a vertex of $G$. Let $n = |V(G)| - 1$ and let the polynomial ring be
\[
\mathcal{A} = K[x_1,\dots,x_a, x_{a+1},\dots, x_{a+n}].
\]
Set $I = I_b(H)$. For each $k \in \{1,\dots,a-b\}$, define the monomial
\[
m_k = x_k x_{k+1} \cdots x_{k+b}.
\]
These are exactly the generators of $I$ corresponding to paths of length $b$ that are entirely contained in the attached path $P_a$. Because $a \ge 2b+1$, the first $b+1$ such monomials $m_1,\dots,m_{b+1}$ are well defined.

Observe that for each $i \in \{1,\dots,b+1\}$, the variable $x_i$ appears in the minimal generating set $\mathcal{G}(I)$ only in the monomials $m_1,\dots,m_i$. Moreover, if $u \in \mathcal{G}(I) \setminus \{m_1,\dots,m_{b+1}\}$, then
\[
\supp(u) \cap \{x_1,\dots,x_{b+1}\} = \emptyset.
\]
Indeed, any path of length $b$ in $H$ that is not contained in $P_a$ must pass through the vertex $a$, which is adjacent to $a-1$. Since the distance from $a$ to vertex $1$ is $a-1$, and $b \le \lfloor (a-1)/2 \rfloor$, such a path cannot reach the initial segment $\{1,\dots,b+1\}$. Thus the variables $x_1,\dots,x_{b+1}$ are confined to the generators $m_1,\dots,m_{b+1}$.

We claim that the monomial
\[
m_1 = x_1 x_2 \cdots x_{b+1}
\]
is a strongly superficial element of $I$. By definition, this means that for every integer $p \ge 1$,
\[
(I^{p+1} : m_1) = I^p.
\]
The inclusion $I^p \subseteq (I^{p+1} : m_1)$ is immediate because $m_1 \in I$. To establish the reverse inclusion, let $f \in (I^{p+1} : m_1)$ be a monomial. (Since the colon ideal of a monomial ideal is monomial, it is enough to consider monomials $f$.) Then $f m_1 \in I^{p+1}$, so there exist a monomial $g \in \mathcal{A}$ and generators $h_1,\dots,h_{p+1} \in \mathcal{G}(I)$ such that
\begin{equation}\label{eq:factor}
f m_1 = g \cdot h_1 h_2 \cdots h_{p+1}.
\end{equation}

If $h_j = m_1$ for some $j$, then cancelling $m_1$ yields $f \in I^p$. Therefore we assume that $m_1 \notin \{h_1,\dots,h_{p+1}\}$.

Since $x_1$ appears only in the generator $m_1$, the left-hand side of \eqref{eq:factor} is divisible by $x_1$, while among the $h_j$ none equals $m_1$. Hence $x_1$ must divide the coefficient monomial $g$. Write $g = x_1 g_1$. Cancelling $x_1$ yields
\begin{equation}\label{eq:step1}
f \, x_2 \cdots x_{b+1} = g_1 \cdot h_1 \cdots h_{p+1}.
\end{equation}

We now proceed by induction on $i \in \{2,\dots,b+1\}$. Suppose that for some $i$ we have constructed an expression
\begin{equation}\label{eq:inductive}
f \, x_i x_{i+1} \cdots x_{b+1} = g_{i-1} \cdot H_1 \cdots H_{p+1},
\end{equation}
where each $H_j \in \mathcal{G}(I)$ and none of the $H_j$ belongs to $\{m_1,\dots,m_{i-1}\}$. (The base case $i=2$ is precisely \eqref{eq:step1} with $H_j = h_j$, since $m_1$ is excluded by assumption.)

Consider the variable $x_i$ on the left-hand side. By the support observation, the only generators in $\mathcal{G}(I)$ that contain $x_i$ are $m_1,\dots,m_i$. Because the $H_j$'s avoid $m_1,\dots,m_{i-1}$, the only possible generator among them that can provide the factor $x_i$ is $m_i$. We distinguish two cases.

\smallskip
\noindent \textit{Case 1:} $m_i \in \{H_1,\dots,H_{p+1}\}$.
Without loss of generality, assume $H_1 = m_i = x_i x_{i+1} \cdots x_{i+b}$. Substituting this into \eqref{eq:inductive} gives
\[
f \, x_i x_{i+1} \cdots x_{b+1} = g_{i-1} \cdot (x_i x_{i+1} \cdots x_{i+b}) \cdot H_2 \cdots H_{p+1}.
\]
Since $i+b > b+1$, the monomial $x_{b+2} \cdots x_{i+b}$ is well defined. Cancelling the common factor $x_i x_{i+1} \cdots x_{b+1}$ yields
\[
f = \bigl( g_{i-1} \cdot x_{b+2} \cdots x_{i+b} \bigr) \cdot H_2 \cdots H_{p+1}.
\]
The right-hand side is a monomial times a product of $p$ minimal generators of $I$. Hence $f \in I^p$, and the induction stops.

\smallskip
\noindent \textit{Case 2:} $m_i \notin \{H_1,\dots,H_{p+1}\}$.
In this case, none of the $H_j$ contains $x_i$. Consequently, the variable $x_i$ on the left-hand side of \eqref{eq:inductive} must divide the coefficient monomial $g_{i-1}$. Write $g_{i-1} = x_i g_i$ and obtain
\[
f \, x_{i+1} \cdots x_{b+1} = g_i \cdot H_1 \cdots H_{p+1}.
\]
This is exactly the form \eqref{eq:inductive} with $i$ replaced by $i+1$, and the induction hypothesis is preserved.

The induction proceeds for $i = 2, 3, \dots, b+1$. If at any step Case~1 occurs, we obtain $f \in I^p$ immediately. If Case~2 persists until $i = b+1$, we reach the equation
\[
f = g_{b+1} \cdot H_1 \cdots H_{p+1},
\]
which directly implies $f \in I^{p+1} \subseteq I^p$. In all cases, we conclude $f \in I^p$. Therefore
\[
(I^{p+1} : m_1) \subseteq I^p,
\]
and equality follows. Thus $m_1$ is a strongly superficial element of $I$, completing the proof.
\end{proof}

\section{Strong Persistence Property of Some Classes of Closed Neighborhood Ideals} \label{SP of neighborhood ideal}

\subsection{Closed Neighborhood Ideals of Graphs with a Degree-One Vertex}

\begin{theorem}\label{degree-one-vertex}
Let $G$ be a graph containing a vertex of degree one. Then the closed neighborhood ideal $CNI(G)$ admits a strongly superficial element of degree $1$. In particular, $CNI(G)$ satisfies the strong persistence property.
\end{theorem}

\begin{proof}
Let $v$ be a leaf of $G$ and let $u$ be its unique neighbor. Then $v$ appears only in the closed neighborhoods $N[v]$ and $N[u]$, and clearly $u_v=x_vx_u$ divides $u_u$. Hence $u_v$ is a minimal generator of $CNI(G)$.

Let $CNI(G)=I$ and set $g_v=x_vx_u$. We claim that
\[
(I^{k+1}:g_v)=I^k \quad \text{for all } k\ge 1.
\]

The inclusion $I^k \subseteq (I^{k+1}:g_v)$ is immediate, since $g_v\in I$.

For the reverse inclusion, let $m\in (I^{k+1}:g_v)$. Then $mg_v\in I^{k+1}$, so there exist minimal generators $g_1,\dots,g_{k+1}\in \mathcal{G}(I)$ and a monomial $f\in \mathcal{A}$ such that
\[
mg_v = f g_1 \cdots g_{k+1}.
\]

If $g_i=g_v$ for some $i$, then canceling $g_v$ yields $m\in I^k$.

Assume now that $g_i\neq g_v$ for all $i$. Since $x_v \mid mg_v$ and no $g_i$ is divisible by $x_v$, it follows that $x_v \mid f$. Write $f=x_v f'$. Then
\[
m x_u = f' g_1 \cdots g_{k+1}.
\]

If $x_u$ divides some $g_i$, say $g_{k+1}=x_u g'$, then canceling $x_u$ gives
\[
m = f' g' g_1 \cdots g_k \in I^k.
\]

Otherwise, $x_u \nmid g_i$ for all $i$, and hence $x_u \mid f'$. Writing $f'=x_u f''$ and canceling $x_u$, we obtain
\[
m = f'' g_1 \cdots g_{k+1} \in I^{k+1} \subseteq I^k.
\]

Thus $(I^{k+1}:g_v)\subseteq I^k$, proving the claim. Therefore, $g_v$ is a strongly superficial element of degree $1$, and the result follows.
\end{proof}

\subsection{Separable Closed Neighborhood Ideals of Graphs}
\begin{definition}
Let \(\mathcal{A} = K[x_1,\dots,x_n]\) be a polynomial ring over a field \(K\), and let \(I \subseteq \mathcal{A}\) be a monomial ideal with unique minimal set of monomial generators \(\mathcal{G}(I) = \{u_1,\dots,u_m\}\).  
We say that \(I\) is \emph{separable} if there exists an index \(i\) (with \(1 \le i \le m\)) and monomials \(w, g \in \mathcal{A}\) such that
\begin{itemize}
    \item \(w \neq 1\),
    \item \(\gcd(w,g) = 1\),
    \item \(u_i = w \cdot g\),
    \item for every \(j \neq i\), \(\gcd(u_j, u_i) = w\).
\end{itemize}
The generator \(u_i\) is called a \emph{split generator} of the separable ideal.
\end{definition}

\begin{theorem} \cite[Theorem 2.11]{P24}
    Every separable ideal has the strong persistence property.
\end{theorem}

\begin{theorem}\label{separable}
Let \(G\) be a finite simple graph whose vertex set decomposes as
\[
V(G)=C \sqcup \{v\} \sqcup X,
\]
where \(C\) is a nonempty clique. Assume that every vertex of \(C\) is adjacent to \(v\) and to every vertex of \(X\), while the vertex \(v\) has no neighbors in \(X\). Then the closed neighborhood ideal \(CNI(G)\) of $G$ is separable, and hence satisfies the strong persistence property. Moreover, if \(X\neq \varnothing\), then the monomial corresponding to \(v\),
\[
u_v=\left(\prod_{c\in C} x_c\right)x_v,
\]
is a split generator of \(CNI(G)\).
\end{theorem}

\begin{proof}
For each vertex \(y\in V(G)\), let
\[
u_y=\prod_{z\in N[y]} x_z
\]
denote the monomial associated with the closed neighborhood of \(y\). Set
\[
w=\prod_{c\in C} x_c.
\]
Since every vertex of \(C\) is adjacent to \(v\) and \(v\) is not adjacent to any vertex of \(X\), we have
\[
N[v]=C\cup\{v\},
\qquad\text{and hence}\qquad
u_v=wx_v.
\]

Now let \(x\in X\). By hypothesis, every vertex of \(C\) is adjacent to \(x\), while \(v\notin N[x]\). Therefore
\[
N[x]=C\cup N_X[x],
\]
where \(N_X[x]\) denotes the closed neighborhood of \(x\) in the induced subgraph \(G[X]\). It follows that
\[
u_x
=
\left(\prod_{c\in C} x_c\right)
\left(\prod_{t\in N_X[x]} x_t\right)
=
w\cdot \prod_{t\in N_X[x]} x_t.
\]
In particular, \(w\mid u_x\) and \(x_v\nmid u_x\), so
\[
\gcd(u_v,u_x)=w
\qquad\text{for every } x\in X.
\]

Next, let \(c\in C\). Since \(C\) is a clique and each vertex of \(C\) is adjacent to every vertex of \(X\) and to \(v\), we obtain
\[
N[c]=V(G),
\]
and thus
\[
u_c=
\left(\prod_{c'\in C} x_{c'}\right)x_v\left(\prod_{x\in X} x_x\right).
\]
Hence \(u_v\mid u_c\), so no generator arising from a vertex of \(C\) is minimal. Consequently, the minimal monomial generators of \(CNI(G)\) consist of \(u_v\) together with the monomials \(u_x\) for \(x\in X\), after removing any repetitions.

Therefore \(u_v=wx_v\) with \(\gcd(w,x_v)=1\), and for every other minimal generator \(u_x\) we have \(\gcd(u_v,u_x)=w\). This is exactly the separability condition. Hence \(CNI(G)\) is separable, and \(u_v\) is a split generator.

If \(X=\varnothing\), then \(CNI(G)=(u_v)\) is principal, and the separability condition is vacuous.
\end{proof}
We conclude this section by presenting an example that highlights the construction in Theorem~\ref{separable}.
\begin{example}
Let $G$ be the graph obtained from the complete graph $K_4$ on the vertices 
$V(G)=\{v_1,v_2,v_3,v_4\}$ by deleting the edge $\{v_1,v_2\}$. Then
\[
N[v_1]=\{v_1,v_3,v_4\}, \quad
N[v_2]=\{v_2,v_3,v_4\}, \quad
N[v_3]=N[v_4]=V(G).
\]
Hence, the minimal generators of the closed neighborhood ideal are
\[
CNI(G)=\langle x_{v_1}x_{v_3}x_{v_4}, \; x_{v_2}x_{v_3}x_{v_4} \rangle.
\]

Set $w=x_{v_3}x_{v_4}$. Then
\[
u_{v_1}=w\,x_{v_1}, \qquad u_{v_2}=w\,x_{v_2},
\]
with $\gcd(w,x_{v_1})=\gcd(w,x_{v_2})=1$, and
\[
\gcd(u_{v_1},u_{v_2})=w.
\]
Thus $u_{v_1}$ (and similarly $u_{v_2}$) is a split generator, and hence $CNI(G)$ is separable.

Moreover, this example fits the framework of Theorem~\ref{separable} by taking 
\[
C=\{v_3,v_4\}, \quad v=v_1, \quad X=\{v_2\},
\]
where $C$ is a clique, $v$ is adjacent to every vertex of $C$, and $v$ is not adjacent to $X$. This illustrates the theorem.
\end{example}

\section{Edge Ideal and Closed Neighborhood Ideal of Helm Graph} \label{helm}

The \emph{edge ideal} of a graph \(G\), introduced by Villarreal \cite{SVV}, is the monomial ideal
\[
I(G)=\left(x_ix_j:\{x_i,x_j\}\in E(G)\right).
\]

\noindent Let \(n\ge 3\), and let \(H_{0,2n}\) denote the Helm graph with vertex set
\[
V(H_{0,2n})=\{x_0,x_1,\ldots,x_n,x_{n+1},\ldots,x_{2n}\},
\]
and edge set
\[
E(H_{0,2n})
=
E(C_n)
\cup
\{\{x_0,x_i\}:1\le i\le n\}
\cup
\{\{x_i,x_{n+i}\}:1\le i\le n\},
\]
where \(E(C_n)\) denotes the edge set of the cycle \(C_n\).

Let $\mathcal{A}=K[x_0,x_1,\ldots,x_{2n}]$ be the polynomial ring over a field \(K\). Then the edge ideal of \(H_{0,2n}\) is
\[
I(H_{0,2n})
=
\left(
x_ix_{i+1}:1\le i\le n-1,\;
x_nx_1,\;
x_0x_i:1\le i\le n,\;
x_ix_{n+i}:1\le i\le n
\right),
\]

and the closed neighborhood ideal of \(H_{0,2n}\) is given by
\[
CNI(H_{0,2n})
=
\left(
x_1x_{n+1},
x_2x_{n+2},
\ldots,
x_nx_{2n},
x_0x_1x_2\cdots x_n
\right).
\]

\noindent Throughout the section, we denote the closed neighborhood ideal of the Helm graph by $L=CNI(H_{0,2n})$.
The following graph is a Helm Graph.

\begin{figure}[H] \centering
\includegraphics[width=7cm, height=6cm]{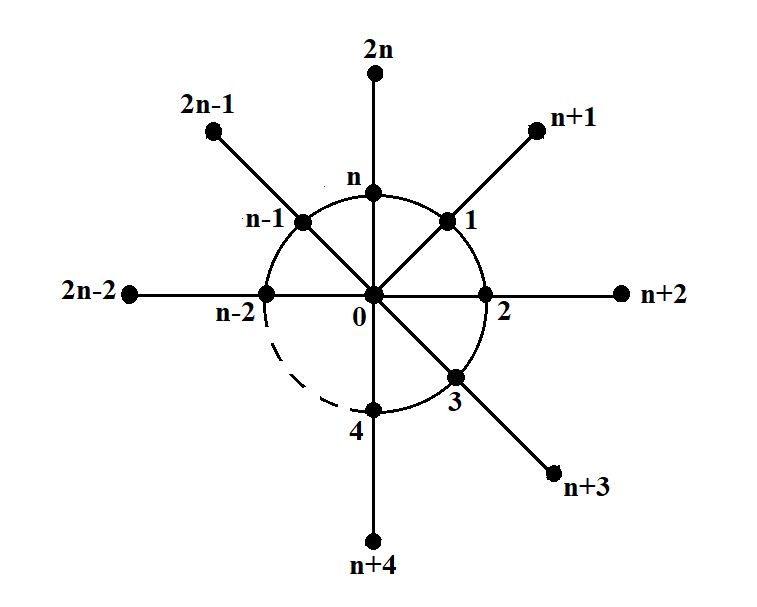}
\caption{Graph $H_{0,2n}$}\label{Helmfig} \end{figure}

\begin{theorem}\label{edge helm}
Let $\mathcal{A}=K[x_0,x_1,\ldots,x_{2n}]$ be a polynomial ring over a field $K$, and let $I(H_{0,2n})$ be the edge ideal of the Helm graph $H_{0,2n}$. Then, for every $n\geq 3$,
\[
\depth(\mathcal{A}/I(H_{0,2n}))=n.
\]
\end{theorem}

\begin{proof}
Consider the following short exact sequence:
\[
0
\longrightarrow
\mathcal{A}/\big(I(H_{0,2n}):x_0\big)
\xrightarrow{\cdot x_0}
\mathcal{A}/I(H_{0,2n})
\longrightarrow
\mathcal{A}/\big(I(H_{0,2n}),x_0\big)
\longrightarrow
0.
\]

Recall that
\begin{align*}
I(H_{0,2n})=
(&x_0x_1,\;x_0x_2,\;x_0x_3,\;\ldots,\;x_0x_{n-2},\;x_0x_{n-1},\;x_0x_n,\\
&x_1x_2,\;x_2x_3,\;x_3x_4,\;\ldots,\;x_{n-2}x_{n-1},\;x_{n-1}x_n,\;x_nx_1,\\
&x_1x_{n+1},\;x_2x_{n+2},\;x_3x_{n+3},\;\ldots,\;x_{n-2}x_{2n-2},\;x_{n-1}x_{2n-1},\;x_nx_{2n}).
\end{align*}

We have
\begin{align*}
\big(I(H_{0,2n}):x_0\big)
=
(&x_1,\;x_2,\;x_3,\;\ldots,\;x_{n-2},\;x_{n-1},\;x_n,\\
&x_1x_2,\;x_2x_3,\;x_3x_4,\;\ldots,\;x_{n-2}x_{n-1},\;x_{n-1}x_n,\;x_nx_1,\\
&x_1x_{n+1},\;x_2x_{n+2},\;x_3x_{n+3},\;\ldots,\;x_{n-2}x_{2n-2},\;x_{n-1}x_{2n-1},\;x_nx_{2n})\\
=
(&x_1,\;x_2,\;\ldots,\;x_n),
\end{align*}
since every remaining generator is divisible by one of the variables
$x_1,\ldots,x_n$.

Hence,
\[
\mathcal{A}/\big(I(H_{0,2n}):x_0\big)
\cong
K[x_0,x_{n+1},x_{n+2},\ldots,x_{2n}],
\]
which is a polynomial ring in $n+1$ variables. Therefore,
\[
\depth\!\left(\mathcal{A}/\big(I(H_{0,2n}):x_0\big)\right)=n+1.
\]

Next, we compute the depth of
$\mathcal{A}/(I(H_{0,2n}),x_0)$.
Observe that
\begin{align*}
\big(I(H_{0,2n}),x_0\big)
=
(&x_0,\;
x_1x_2,\;x_2x_3,\;\ldots,\;x_{n-2}x_{n-1},\;x_{n-1}x_n,\;x_nx_1,\\
&x_1x_{n+1},\;x_2x_{n+2},\;\ldots,\;x_{n-1}x_{2n-1},\;x_nx_{2n}).
\end{align*}

Hence,
\begin{align*}
\depth\!\left(\mathcal{A}/(I(H_{0,2n}),x_0)\right)
=&\;
\depth\Bigl(
K[x_0,x_1,x_2,\ldots,x_{2n}]
/
(x_0,\,
x_1x_2,\,
x_2x_3,\,
\ldots,
x_{n-1}x_n,\,
x_nx_1,\\
&\hspace{3.5cm}
x_1x_{n+1},\,
x_2x_{n+2},\,
\ldots,\,
x_{n-1}x_{2n-1},\,
x_nx_{2n})
\Bigr)\\
=&\;
\depth(\mathcal{A}'/I(S'_{2n}))\\
=&\;
n,
\end{align*}
where $I(S'_{2n})$ is the edge ideal of the sunlet graph over the polynomial ring 
\(
\mathcal{A}'
=
K[x_1,x_2,\ldots,x_n, \\ x_{n+1},\ldots,x_{2n}].
\)
The last equality follows from \cite{BGIRA}[Theorem 2.1].

Finally, since
\[
\depth\!\left(\mathcal{A}/(I(H_{0,2n}):x_0)\right)=n+1
>
n
=
\depth\!\left(\mathcal{A}/(I(H_{0,2n}),x_0)\right),
\]
an application of \cite[Lemma 3.1.4]{WVb} yields
\[
\depth(\mathcal{A}/I(H_{0,2n}))=n.
\]
This completes the proof.
\end{proof}

\begin{lemma}\label{minhelm}
Let
\(
L=CNI(H_{0,2n})
=
(x_1x_{n+1},x_2x_{n+2},\ldots,x_nx_{2n},x_0x_1x_2\cdots x_n)
\)
be the closed neighborhood ideal of the Helm graph \(H_{0,2n}\). Then the set of minimal primes is given by
\[
\begin{aligned}
\Min(\mathcal A/L)
=
&\left\{
(y_1,\ldots,y_n):
y_i\in\{x_i,x_{n+i}\},
\;
(y_1,\ldots,y_n)\neq(x_{n+1},\ldots,x_{2n})
\right\}
\\
&\cup
\left\{
(x_0,x_{n+1},\ldots,x_{2n})
\right\}.
\end{aligned}
\]
Consequently,
\(
|\Min(\mathcal A/L)|=2^n.
\)
Moreover, \(2^n-1\) of the minimal primes have height \(n\), while the minimal prime
\(
(x_0,x_{n+1},\ldots,x_{2n})
\)
has height \(n+1\).
\end{lemma}

\begin{proof}
By \cite[Proposition~2.6]{honey}, the minimal primes of the closed neighborhood ideal \(L\) are in one-to-one correspondence with the minimal dominating sets of the Helm graph \(H_{0,2n}\).

\noindent Observe that every minimal dominating set of \(H_{0,2n}\) is obtained by selecting exactly one vertex from each pair
\(
\{x_i,x_{n+i}\},\,\, 1\le i\le n,
\)
except for the selection
\(
\{x_{n+1},x_{n+2},\ldots,x_{2n}\},
\)
which fails to dominate the central vertex \(x_0\). Therefore, the corresponding minimal primes are precisely
\[
(y_1,\ldots,y_n),
\qquad
y_i\in\{x_i,x_{n+i}\},
\]
excluding the prime
\(
(x_{n+1},\ldots,x_{2n}).
\)
The excluded selection gives rise to the prime
\[
(x_0,x_{n+1},\ldots,x_{2n}),
\]
which is also a minimal prime of \(L\). Consequently,
\[
\begin{aligned}
\Min(\mathcal A/L)
=
&\left\{
(y_1,\ldots,y_n):
y_i\in\{x_i,x_{n+i}\},
\;
(y_1,\ldots,y_n)\neq(x_{n+1},\ldots,x_{2n})
\right\}
\\
&
\cup
\left\{
(x_0,x_{n+1},\ldots,x_{2n})
\right\}.
\end{aligned}
\]

\noindent There are \(2^n\) ways to choose exactly one vertex from each pair
\(\{x_i,x_{n+i}\}\). Excluding one choice and replacing it by
\((x_0,x_{n+1},\ldots,x_{2n})\) yields
\[
|\Min(\mathcal A/L)|
=
(2^n-1)+1
=
2^n.
\]

\noindent Finally, each prime of the form
\((y_1,\ldots,y_n)\) is generated by \(n\) variables and hence has height \(n\), whereas the prime
\((x_0,x_{n+1},\ldots,x_{2n})\) is generated by \(n+1\) variables and therefore has height \(n+1\). This completes the proof.
\end{proof}

\begin{example}
Consider the Helm graph $H_{0,6}$ and its closed neighborhood ideal
\[
L=CNI(H_{0,6})
=
(x_0x_1x_2x_3,\;
x_1x_4,\;
x_2x_5,\;
x_3x_6).
\]
The minimal prime ideals of $L$ are
\[
\begin{aligned}
&(x_1,x_2,x_3),\;
(x_1,x_2,x_6),\;
(x_1,x_3,x_5),\;
(x_1,x_5,x_6),\\
&(x_2,x_3,x_4),\;
(x_2,x_4,x_6),\;
(x_3,x_4,x_5),\;
(x_0,x_4,x_5,x_6).
\end{aligned}
\]

\noindent Observe that each minimal prime, except $(x_0,x_4,x_5,x_6)$, is obtained by choosing exactly one vertex from each pair
\[
\{x_1,x_4\},\qquad
\{x_2,x_5\},\qquad
\{x_3,x_6\},
\]
provided that not all pendant vertices are selected simultaneously. Indeed, the choice $\{x_4,x_5,x_6\}$
does not dominate the central vertex $x_0$, and therefore the corresponding minimal prime is $(x_0,x_4,x_5,x_6).$ \\
This example illustrates the correspondence between the minimal primes of $L$ and the minimal dominating sets of the Helm graph $H_{0,2n}$, which is established in Lemma~\ref{minhelm}.
\end{example}

\begin{lemma}\label{dimhelm}
Let
\(
\mathcal{A}=K[x_0,x_1,\ldots,x_n,x_{n+1},\ldots,x_{2n}]
\)
be a polynomial ring over a field \(K\), and let
\(
L=CNI(H_{0,2n})
\)
be the closed neighborhood ideal of the Helm graph \(H_{0,2n}\). Then $\operatorname{ht}(L)=n$
and $\dim(\mathcal{A}/L)=n+1.$
\end{lemma}

\begin{proof}
By Lemma~\ref{minhelm}, the minimal primes of \(L\) consist of \(2^n-1\) primes of height \(n\) together with one minimal prime of height \(n+1\). Therefore,
\[
\operatorname{ht}(L)
=
\min\{\operatorname{ht}(P):P\in\Min(\mathcal{A}/L)\}
=
n.
\]

\noindent Since \(\mathcal{A}\) is a polynomial ring in \(2n+1\) variables, we have $\dim(\mathcal{A})=2n+1.$ Hence, by the dimension formula,
\[
\dim(\mathcal{A}/L)
=
\dim(\mathcal{A})-\operatorname{ht}(L)
=
(2n+1)-n
=
n+1.
\]
\end{proof}

Now we will compute the depth of closed neighborhood ideals of helm graph by using the dimension of $\mathcal{A}/L.$ 
\begin{theorem}\label{closed helm}
Let
\(
\mathcal{A}=K[x_0,x_1,\ldots,x_{2n}]
\)
be a polynomial ring over a field \(K\), and let
\(
L=CNI(H_{0,2n}).
\)
Then
\(
\depth(\mathcal{A}/L)=n.
\)
\end{theorem}

\begin{proof}
Recall that
\[
L=CNI(H_{0,2n})
=
(x_1x_{n+1},x_2x_{n+2},\ldots,x_nx_{2n},x_0x_1x_2\cdots x_n).
\]
Let
\(
J=(x_1x_{n+1},x_2x_{n+2},\ldots,x_nx_{2n})
\)
and
\(
u=x_0x_1x_2\cdots x_n.
\)
Then
\(
L=J+(u).
\)

Consider the following short exact sequence:
\[
0
\longrightarrow
\mathcal{A}/(J:u)
\xrightarrow{\cdot u}
\mathcal{A}/J
\longrightarrow
\mathcal{A}/L
\longrightarrow
0.
\]

\noindent We first compute the depth of the quotient ring \(\mathcal{A}/J\). Observe that the generators
\[
x_1x_{n+1},\;
x_2x_{n+2},\;
\ldots,\;
x_nx_{2n}
\]
have pairwise disjoint supports. Consequently, they form a regular sequence in $\mathcal{A}$. It follows from \cite{BH93} that \(J\) is a complete intersection ideal of height \(n\). Hence, the quotient ring $(\mathcal{A}/J)$ is a Cohen-Macaulay ring. Therefore,
\[
\depth(\mathcal{A}/J)
=
\dim(\mathcal{A}/J)
=
(2n+1)-n
=
n+1.
\]

\noindent Next, we compute the colon ideal \(J:u\). Since
\(
u=x_0x_1x_2\cdots x_n,
\)
for each \(i=1,\ldots,n\),
\[
\frac{x_ix_{n+i}}{\gcd(x_ix_{n+i},u)}
=
\frac{x_ix_{n+i}}{x_i}
=
x_{n+i}.
\]
Hence,
\(
J:u=(x_{n+1},x_{n+2},\ldots,x_{2n}).
\)
Therefore,
\(
\mathcal{A}/(J:u)
\cong
K[x_0,x_1,\ldots,x_n],
\)
which is a polynomial ring in \(n+1\) variables. Thus,
\[
\depth(\mathcal{A}/(J:u))
=
n+1.
\]

\noindent Applying the Depth Lemma \cite[Lemma~3.1.4]{WVb} to the above short exact sequence, we obtain
\[
\depth(\mathcal{A}/L)\ge n.
\]

\noindent On the other hand, by Lemma~\ref{minhelm}, the ideal \(L\) has \(2^n-1\) minimal primes of height \(n\) and one minimal prime of height \(n+1\). Hence \(L\) is not unmixed. Since by \cite[Theorem~5.3.16]{FMRS18} every Cohen--Macaulay ring is unmixed, it follows that \(\mathcal{A}/L\) is not Cohen--Macaulay. Therefore,
\(
\depth(\mathcal{A}/L)
<
\dim(\mathcal{A}/L).
\)
By Lemma~\ref{dimhelm},
\(
\dim(\mathcal{A}/L)=n+1.
\)
Hence,
\[
\depth(\mathcal{A}/L)\le n.
\]

Combining the inequalities
\[
\depth(\mathcal{A}/L)\ge n
\quad\text{and}\quad
\depth(\mathcal{A}/L)\le n,
\]
we conclude that
\[
\depth(\mathcal{A}/L)=n.
\]
\end{proof}

For a monomial ideal $I\subseteq\mathcal{A}$, we denote by
$\operatorname{Min}(I)$ the set of minimal prime ideals of $I$. If $I$ is a square-free monomial ideal, then $\operatorname{Ass}(\mathcal{A}/I)=\operatorname{Min}(I)$, \cite[Corollary~1.3.6]{bookHerzog}). Therefore, for square-free monomial ideals, the $v$-number can be expressed as
\[
v(I)
=
\min\{\deg(f):f\in\mathcal{A},\ (I:f)\in
\operatorname{Min}(I)\}.
\]
\begin{lemma} \label{vnumb}
 Let $L=CNI(H_{0,2n})=(x_1x_{n+1},x_2x_{n+2},\ldots,x_nx_{2n},x_0x_1x_2\cdots x_n)
\subseteq
\mathcal{A}.$ If \(g\in\mathcal{A}\) is a monomial satisfying $\deg(g)<n,$ then $(L:g)\notin\operatorname{Ass}(\mathcal{A}/L).$\end{lemma}

\begin{proof}
Suppose that \(\deg(g)<n\). Then \(g\) cannot contain one variable from every pair $\{x_i,x_{n+i}\}; i=1,\ldots,n.$ Hence there exists an index \(i\) such that neither \(x_i\) nor \(x_{n+i}\) divides \(g\).

If \((L:g)\) were an associated prime of \(L\), then it must coincide with one of the minimal primes of \(L\) described in Lemma~\ref{minhelm}.

In the first case of Lemma~\ref{dimhelm}, every minimal associated prime contains exactly one variable from each pair \(\{x_i,x_{n+i}\}\). Therefore, for the above index \(i\), either \(x_i\in(L:g)\) or \(x_{n+i}\in(L:g)\). This would imply that either \(x_ig\in L\) or \(x_{n+i}g\in L\). Since \(g\) contains neither \(x_i\) nor \(x_{n+i}\), neither product is divisible by any minimal generator of \(L\), which is impossible.

In the second case, $(L:g)=(x_0,x_{n+1},\ldots,x_{2n}).$ Since \(\deg(g)<n\), the monomial \(g\) omits at least one variable among \(x_1,\ldots,x_n\). Consequently, for some \(i\), $x_{n+i}g\notin L,$ contradicting the assumption that \(x_{n+i}\in(L:g)\).

Therefore, $(L:g)\notin\operatorname{Ass}(\mathcal{A}/L)$ whenever \(\deg(g)<n\).
\end{proof}

\begin{theorem} \label{vnumber helm}
Let $L=CNI(H_{0,2n}) \subset \mathcal{A}=\mathbb{K}[x_0,x_1 \ldots,x_n,x_{n+1},\ldots,x_{2n}]$ be the closed neighborhood ideal of the Helm graph \(H_{0,2n}\).
Then $v(L)=n.$
\end{theorem}

\begin{proof}
By Lemma~\ref{minhelm}, the associated primes of $L$ consist of all minimal primes of the form 
\[
(x_0,x_{n+1}, x_{n+2}, \ldots, x_{2n}) \quad \text{and} \quad (y_1,\ldots,y_n);
y_i\in\{x_i,x_{n+i}\}, 
\]
excluding the choice $(x_{n+1},x_{n+2},\ldots,x_{2n}).$

\noindent Consider the monomial $f=x_{n+1}x_{n+2}\cdots x_{2n}.$ For each $i=1,\ldots,n$,
\[
x_if=(x_ix_{n+i})\prod_{j\neq i}x_{n+j}\in L,
\]
while
\[
x_{n+i}f\notin L
\quad\text{and}\quad
x_0f\notin L.
\]
Hence $(L:f)=(x_1,x_2,\ldots,x_n),$ which is a minimal associated prime of $L$. Therefore,
\[
v(L)\leq\deg(f)=n.
\]

\noindent On the other hand, by the lemma \ref{vnumb}, if $g$ is any monomial with $\deg(g)<n,$ then 
\[
(L:g)\notin\operatorname{Ass}(\mathcal{A}/L).
\]
Thus no monomial of degree strictly less than $n$ contributes to the definition of the $v$-number. Consequently,
\[
v(L)\geq n.
\]

\noindent Combining the inequalities, we conclude that $v(L)=n.$
\end{proof}






\begin{thebibliography}{99}







\bibitem{BGIRA}
H.~Bibi, H.~Garmini, Irawati, A.~Rauf, and Um-E-Aimen,
Some algebraic invariants of ideals associated with graph families.
To appear in \emph{Electron. J. Graph Theory Appl.}

\bibitem{BRU24}
H.~Bibi, A.~Rauf, and A.~Umar,
On powers of some classes of monomial ideals.
\emph{Bull. Korean Math. Soc.} \textbf{61} (2024), no.~6, 1579--1591.

\bibitem{biswas}
A.~Biswas and M.~Mandal,
A study of $v$-number for some monomial ideals.
2023, \arxiv{2308.08604}.


\bibitem{Bro79}
M.~Brodmann,
Asymptotic stability of $\operatorname{Ass}(M/I^nM)$.
\emph{Proc. Amer. Math. Soc.} \textbf{74} (1979), no.~1, 16--18.

\bibitem{2Bro79}
M.~Brodmann,
The asymptotic nature of the analytic spread.
\emph{Math. Proc. Cambridge Philos. Soc.} \textbf{86} (1979), no.~1, 35--39.

\bibitem{BH93}
W.~Bruns and J.~Herzog,
\emph{Cohen--Macaulay Rings}.
Cambridge Stud. Adv. Math. 39,
Cambridge University Press, Cambridge, 1993.

\bibitem{vnumber}
S.~M. Cooper, A.~Seceleanu, S.~O. Toh\u{a}neanu, M.~Vaz Pinto, and R.~H.~Villarreal,
Generalized minimum distance functions and algebraic invariants of Geramita ideals.
\emph{Adv. Appl. Math.} \textbf{112} (2020), 101940.

\bibitem{con+99}
A. Conca and E. De Negri, $M$-sequences, graph ideals, and ladder ideals of linear type, \textit{J. Algebra}, \textbf{211} (1999), 599--624.

\bibitem{Mac}
D.~R. Grayson and M.~E. Stillman,
\emph{Macaulay2, a software system for research in algebraic geometry}.
2002,
\url{http://www.math.uiuc.edu/Macaulay2/}.

\bibitem{FHT11}
C.~A.~Francisco, H.~T.~H\`a, and A.~Van Tuyl,
Coloring of hypergraphs, perfect graphs and associated primes of powers of monomial ideals.
\emph{J. Algebra} \textbf{331} (2011), 224--242.

\bibitem{S}
G.~M. Greuel, G.~Pfister, and H.~Sch\"onemann,
\emph{Singular 2.0: A Computer Algebra System for Polynomial Computations}.
Centre for Computer Algebra, University of Kaiserslautern, 2001,
\url{http://www.singular.uni-kl.de}.

\bibitem{bookHerzog}
J.~Herzog and T.~Hibi,
\emph{Monomial Ideals}.
Grad. Texts in Math. 260,
Springer, London, 2011.

\bibitem{HQ15}
J.~Herzog and A.~A.~Qureshi,
Persistence and stability properties of powers of ideals.
\emph{J. Pure Appl. Algebra} \textbf{219} (2015), no.~3, 530--542.

\bibitem{HRV13}
J.~Herzog, A.~Rauf, and M.~Vladoiu,
The stable set of associated prime ideals of a polymatroidal ideal.
\emph{J. Algebr. Comb.} \textbf{37} (2013), no.~2, 289--312.

\bibitem{honey}
J.~Honeycutt and K.~Sather-Wagstaff,
Closed neighborhood ideals of finite simple graphs.
\emph{La Matematica} \textbf{1} (2022), 387--394.

\bibitem{JV}
D.~Jaramillo and R.~H.~Villarreal,
The $v$-number of edge ideals.
\emph{J. Combin. Theory Ser. A} \textbf{177} (2021), 105310.


\bibitem{MMV12}
J.~Mart\'inez-Bernal, S.~Morey, and R.~H.~Villarreal,
Associated primes of powers of edge ideals.
\emph{Collect. Math.} \textbf{63} (2012), no.~3, 361--374.

\bibitem{Susan}
S.~Morey,
Depths of powers of the edge ideal of a tree.
\emph{Comm. Algebra} \textbf{38} (2010), no.~11, 4042--4055.

\bibitem{Mor+26}
S.~Moradi and L.~Sharifan,
On homological invariants and Cohen--Macaulayness of closed neighborhood ideals.
2026, \arxiv{2602.07910}

\bibitem{FMRS18}
W.~F. Moore, M.~Rogers, and K.~Sather-Wagstaff,
\emph{Monomial Ideals and Their Decompositions}.
Universitext, Springer, Cham, 2018.

\bibitem{N1}
M.~Nasernejad, A.~A.~Qureshi, S.~Bandari, and A.~Musapa{\c{s}}ao{\u{g}}lu,
Dominating ideals and closed neighborhood ideals of graphs.
\emph{Mediterr. J. Math.} \textbf{19} (2022), Article~152.

\bibitem{Nas14}
M.~Nasernejad,
Asymptotic behavior of associated primes of monomial ideals with combinatorial applications.
\emph{J. Algebra Relat. Top.} \textbf{2} (2014), no.~1, 15--25.

\bibitem{P24}
M.~Nasernejad,
Persistence property for some classes of monomial ideals of a polynomial ring.
\emph{J. Algebra Appl.} \textbf{16} (2017), no.~6, 1750105.

\bibitem{Nas+22}
M.~Nasernejad, A.~A.~Qureshi, S.~Bandari, and A.~Musap{\c{s}}ao\u{g}lu,
Dominating ideals and closed neighborhood ideals of graphs.
\emph{Mediterr. J. Math.} \textbf{19} (2022), 1--18.

\bibitem{Nas+25}
M.~Nasernejad, S.~Bandari, and L.~G.~Roberts,
Normality and associated primes of closed neighborhood ideals and dominating ideals.
\emph{J. Algebra Appl.} \textbf{24} (2025), no.~1, 2550009.

\bibitem{Olt+26}
A.~Olteanu and O.~Olteanu,
On the closed neighborhood ideal of the square of the path graph.
2026, \arxiv{2602.14163}

\bibitem{RNA18}
S.~Rajaee, M.~Nasernejad, and I.~Al-Ayyoub,
Superficial ideals for monomial ideals.
\emph{J. Algebra Appl.} \textbf{17} (2018), no.~6, 1850102.

\bibitem{MCf}
A.~Rauf,
Depth and Stanley depth of multigraded modules.
\emph{Comm. Algebra} \textbf{38} (2010), no.~2, 773--784.

\bibitem{RT17}
E.~Reyes and J.~Toledo,
On the strong persistence property for monomial ideals.
\emph{Bull. Math. Soc. Sci. Math. Roumanie (N.S.)}
\textbf{60} (2017), no.~3, 293--305.

\bibitem{Sha+20}
L.~Sharifan and S.~Moradi,
Closed neighborhood ideal of a graph.
\emph{Rocky Mountain J. Math.} \textbf{50} (2020), no.~3, 1097--1107.

\bibitem{SVV}
A.~Simis, W.~V.~Vasconcelos, and R.~H.~Villarreal,
On the ideal theory of graphs.
\emph{J. Algebra} \textbf{167} (1994), no.~2, 389--416.

\bibitem{WVb}
W.~V. Vasconcelos,
\emph{Arithmetic of Blowup Algebras}.
London Math. Soc. Lecture Note Ser. 195,
Cambridge University Press, Cambridge, 1994.

\bibitem{Zhu16}
G.~Zhu,
Depth and Stanley depth of the path ideal associated to an $n$-cyclic graph.
\emph{Turkish J. Math.} \textbf{41} (2017), no.~5, 1240--1254.

\end{thebibliography}
\end{document}